\documentclass[12pt,a4paper,leqno]{article}
\usepackage{amsmath,amsfonts,amssymb,amscd}
\usepackage[margin=2cm]{geometry}

\usepackage{theorem}
\usepackage{url}
\usepackage{hyperref}

\theoremstyle{change}
\theorembodyfont{\sl}
\newtheorem{Theorem}[subsection]{Theorem}

\newtheorem{Lemma}[subsection]{Lemma}

\theorembodyfont{\rmfamily}
\newtheorem{Definition}[subsection]{Definition}
\newtheorem{Remark}[subsection]{Remark}

\makeatletter
\let\c@equation=\c@subsection

\makeatother

\newenvironment{prf}[1]{\trivlist
\item[\hskip \labelsep{\it
#1\hspace*{.3em}}]}{~\hspace{\fill}~$\square$\endtrivlist}

\newenvironment{proof}{\begin{prf}{\bf Proof}}{\end{prf}}

\newcommand{\ol}{\overline}

\newcommand{\C}{\mathbb C}
\newcommand{\F}{\mathbb F}
\newcommand{\Fbar}{\ol{\F}}

\newcommand{\Q}{\mathbb Q}
\newcommand{\Qbar}{\ol{\Q}}
\newcommand{\Z}{\mathbb Z}
\newcommand{\Zbar}{\ol{\Z}}
\newcommand{\A}{\mathbb A}
\renewcommand{\P}{\mathbb P}
\newcommand{\calO}{\mathcal{O}}
\newcommand{\End}{\mathrm{End}}
\newcommand{\Aut}{\mathrm{Aut}}
\newcommand{\GL}{\mathrm{GL}}
\newcommand{\SL}{\mathrm{SL}}
\newcommand{\Spec}{\mathrm{Spec}}
\newcommand{\pr}{\mathrm{pr}}
\newcommand{\Gal}{\mathrm{Gal}}
\newcommand{\Pic}{\mathrm{Pic}}
\newcommand{\red}{\mathrm{red}}
\newcommand{\sep}{\mathrm{sep}}
\newcommand{\eps}{\varepsilon}

\title{A mod $p$ variant of the Andr\'e--Oort conjecture}
\author{Bas  Edixhoven  \& Rodolphe Richard \\
\small{
 \href{mailto:edix@math.leidenuniv.nl}{edix@math.leidenuniv.nl} \quad
 \href{mailto:rodolphe.richard@normalesup.org}{rodolphe.richard@normalesup.org}} 
}

\begin{document}

\maketitle

\begin{abstract}
We state and prove a variant of the Andr\'e--Oort conjecture for the
product of 2 modular curves in positive characteristic, assuming GRH
for quadratic fields.
\end{abstract}

\section{Introduction}


\begin{Definition}\label{def_CM_point}
For a point $x$ in a scheme $X$ we let $\kappa(x)=\calO_{X,x}/m_x$ be
its residue field, and we denote $\iota_x\colon\Spec(\kappa(x))\to X$
the induced $\kappa(x)$-point of~$X$. So we may view $\iota_x$ as an
element of~$X(\kappa(x))$, the set of $\kappa(x)$-valued points
of~$X$. For $X=\A^2$, we have $X(\kappa(x))=\kappa(x)^2$.

By \emph{CM-point in $\A^2_\Q$} we mean a closed point $s$ of the
affine plane over~$\Q$, such that both coordinates of
$\iota_s\in\kappa(s)^2$ are $j$-invariants of CM elliptic curves.

By \emph{CM-point in $\A^2_\Z$} we mean the closure in $\A^2_\Z$ of a
CM-point in~$\A^2_\Q$. We view such a CM-point $\ol{\{s\}}$ as a
closed subset, or as a reduced closed subscheme. For any prime number
$p$ we then denote by $\ol{\{s\}}_{\F_p}$ the reduced fibre over~$p$
and call it the reduction of~$s$ at~$p$.
\end{Definition}

\begin{Theorem}\label{Theorem_main}
  Assume the generalised Riemann hypothesis for quadratic fields.  Let
  $p$ be a prime number. Let $\Sigma$ be a set of finite closed
  subsets $s$ of $\A^2_{\F_p}$ that are reductions of CM-points
  in~$\A^2_\Z$. Let $Z$ be the Zariski closure of the union of all $s$
  in~$\Sigma$. Then every irreducible component of dimension $1$ of
  $Z$ is special: a fibre of one of the 2 projections, or an
  irreducible component of the image in $\A^2_{\F_p}$ of some
  $Y_0(n)_{\F_p}$ with $n\in\Z_{\geq 1}$.
\end{Theorem}

\begin{Remark}
  If $K_1,\ldots,K_n$ are quadratic subfields of~$\Qbar$, then GRH
  holds for their compositum $K$ if and only if it holds for each
  quadratic subfield of~$K$ (the zeta function of~$K$ is the product
  of the Riemann zeta-function and the $L$-functions of the quadratic
  subfields of~$K$).
\end{Remark}

\section{Some facts on CM elliptic curves}
\label{sec:2}
We will need some results on CM elliptic curves and their reduction
mod~$p$. For more detail see~\cite[\S2]{Edixhoven1}, and references
therein. 

For $E$ over $\Qbar$ an elliptic curve with CM, $\End(E)$ is an order
in an imaginary quadratic field~$K$, hence isomorphic to $O_{K,f}=\Z+f
O_K$, with $O_K$ the ring of integers in $K$, and $f\in\Z_{\geq 1}$,
unique, called the conductor. 

For $K\subset \Qbar$ imaginary quadratic and $f\geq 1$, we let
$S_{K,f}$ be the set of isomorphism classes of $(E,\alpha)$, where $E$
is an elliptictic curve over $\Qbar$ and $\alpha\colon
O_{K,f}\to \End(E)$ is an isomorphism, such that the action of
$\End(E)$ on the tangent space of $E$ at $0$ induces the given
embedding $K\to\Qbar$. The group $\Pic(O_{K,f})$ acts on $S_{K,f}$,
making it a torsor. This action commutes with the action of
$G_K:=\Gal(\Qbar/K)$, giving a group morphism $G_K\to\Pic(O_{K,f})$
through which $G_K$ acts on~$S_{K,f}$. This map is surjective,
unramified outside~$f$, and the Frobenius element at a maximal ideal
$m$ of $O_K$ outside $f$ is the class $[m^{-1}]$ in~$\Pic(O_{K,f})$.

For $f'\geq1$ dividing $f$, the inclusion $O_{K,f}\to O_{K_{f'}}$
induces compatible surjective maps of groups $\Pic(O_{K,f})\to
\Pic(O_{K,f'})$ and of torsors $S_{K,f}\to S_{K,f'}$: $(E,\alpha)$ is
mapped to $O_{K,f'}\otimes_{O_{K,f}}E$ with its $O_{K,f'}$-action. In
terms of complex lattices: $O_{K,f'}\otimes_{O_{K,f}}\C/\Lambda =
\C/O_{K,f'}\Lambda$.

For $p$ a prime number, and $f'$ the prime to $p$ part of~$f$, the map
$S_{K,f}\to S_{K,f'}$ is the quotient by the inertia subgroup at any
of the maximal ideals $m$ of $O_K$ containing~$p$ (to show this, use
the adelic description of ramification in class field theory).

Elliptic curves with CM over $\Qbar$ extend uniquely over $\Zbar$ (the
integral closure of $\Z$ in~$\Qbar$), and their endomorphisms as
well. 

For $K$ and $f$ as above we define $j_{K,f}$ to be the image of
$j(E)\colon \Spec(\Zbar)\to \A^1_\Z$, where $E$ is an elliptic curve
over $\Zbar$ with $\End(E)$ isomorphic to~$O_{K,f}$; this does not
depend on the choice of~$E$. Then $j_{K,f}$ is an irreducible closed
subset of $\A^1_\Z$. We equip it with its reduced induced scheme
structure. Then it is finite over $\Z$ of degree~$\#\Pic(O_{K,f})$,
and in fact $j_{K,f}(\Zbar)$ is in bijection with $S_{K,f}$ and hence
is a $\Pic(O_{K,f})$-torsor (here we use that $K$ has a given
embedding into~$\Qbar$).  For $p$ prime, we let $j_{K,f,\F_p}$ be the
fibre of $j_{K,f}$ over~$\F_p$.

Let $p$ be a prime number, and $K$ and $f$ as above. If $p$ is not
split in $O_K$ then $j_{K,f,\F_p}$ consists of supersingular points,
and $j_{K,f}$ can be highly singular above~$p$ (by lack of
supersingular targets). If $p$ is split in $O_K$ then $j_{K,f,\F_p}$
consists of ordinary points, and the corresponding elliptic curves
over $\Fbar_p$ have endomorphism ring isomorphic to $O_{K,f'}$, where
$f'$ is the prime to $p$ part of~$f$, and then
$j_{K,f',\F_p}=(j_{K,f,\F_p})_\red$, and for each morphism of rings
$\Zbar\to\Fbar_p$ the map $j_{K,f'}(\Zbar)\to j_{K,f'}(\Fbar_p)$ is a
bijection and it makes $j_{K,f',\F_p}(\Fbar_p)$ into a
$\Pic(O_{K,f'})$-torsor. Note that every ordinary $x$ in $\Fbar_p$
belongs to exactly one~$j_{K,f'}(\Fbar_p)$.


\section{Some facts about pairs of CM elliptic curves}

Let $s$ be a CM-point in $\A^2_\Q$ as in Def.~\ref{def_CM_point}. Then
$s(\Qbar)$ is a $G_\Q$-orbit. Let $(x_1,x_2)$ be in $s(\Qbar)$. Then
$x_1$ is in $j_{K_1,f_1}(\Qbar)$ for a unique imaginary quadratic
subfield $K_1$ of $\Qbar$, and similarly for~$x_2$, and $G_{K_1K_2}$
acts through $\Pic(O_{K_1,f_1})\times\Pic(O_{K_2,f_2})$, and
$s(\Qbar)$ decomposes into at most $4$ orbits under~$G_{K_1K_2}$.

Let $p$ be a prime. Let $s$ be a finite closed subset of $\A^2_{\F_p}$
that is the reduction at $p$ of a CM-point in~$\A^2_\Z$ (see
Def.~\ref{def_CM_point}). Then $s(\Fbar_p)$ is a finite subset of
$\Fbar_p\times \Fbar_p$ that is stable under
$G_{\F_p}:=\Gal(\Fbar_p/\F_p)$. For each of the $2$ projections, the
image of $s(\Fbar_p)$ consists entirely of ordinary points or entirely
of supersingular points (this follows from the facts recalled
in~\S\ref{sec:2}). If for all $(x_1,x_2)$ in $s(\Fbar_p)$ both $x_1$
and $x_2$ are ordinary, then the $x_1$ form a
$\Pic(O_{K_1,f_1})$-orbit, and the $x_2$ form a
$\Pic(O_{K_2,f_2})$-orbit, with $f_1$ and $f_2$ prime to~$p$.

\section{Images under suitable Hecke correspondences}
For $\ell$ a prime number, $T_\ell$ denotes the correspondence on the
$j$-line, over any field not of characteristic~$\ell$, sending an
elliptic curve $E$ over an algebraically closed field $k$ to the sum
of its $\ell+1$ quotients by the subgroups of $E(k)$ of order~$\ell$.
Similarly, $T_\ell\times T_\ell$ is the correspondence on the $j$-line times
itself that sends a pair of elliptic curves $(E_1,E_2)$ to the sum of
all $(E_1/C_1,E_2/C_2)$ with $C_1$ and $C_2$ subgroups of
order~$\ell$.

\begin{Theorem}\label{thm:primes}
  Assumptions as in Theorem~\ref{Theorem_main}, and assume that all
  irreducible components of $Z$ are of dimension~$1$, and are not a
  fibre of any of the $2$ projections. There are infinitely many prime
  numbers $\ell$ such that $Z\cap (T_\ell\times T_\ell)Z$ is of
  dimension~$1$.
\end{Theorem}
\begin{proof}
There are only finitely many points $(x_1,x_2)$ in $Z(\Fbar_p)$ such
that $x_1$ or $x_2$ is not ordinary. Therefore we can replace $\Sigma$
by its subset of $s$'s whose images under both projections are ordinary.

At this point we combine the arguments of~\cite{Edixhoven2} with
reduction modulo~$p$.  Let $d_1$ and $d_2$ be the degrees of the
projections from $Z$ to $\A^1_{\F_p}$.

For $s$ in $\Sigma$ and $(x_1,x_2)$ in~$s(\Fbar_p)$, let $O_{1,s}$ and
$O_{2,s}$ be the endomorphism rings of the elliptic curves $E_1$ and
$E_2$ over $\Fbar_p$ corresponding to $x_1$ and~$x_2$.

We claim that for all but finitely many $s$ there is a prime number
$\ell$ such that $\ell$ is split in both $O_{1,s}$ and $O_{2,s}$, and
$\# s(\Fbar_p) > 2d_1d_2(\ell+1)^2$, and $\ell > \log(\#
s(\Fbar_p))$. This claim follows, as in the proof of
\cite[Lemma~7.1]{Edixhoven2}, from the (conditional) effective
Chebotarev theorem of Lagarias and Odlyzko \cite{Lag-Odl} as stated in
Theorem~4 of~\cite{Serre_chebotarev}, and Siegel's theorem on class
numbers of imaginary quadratic fields, \cite{Siegel} and
\cite[Ch.~XVI]{Lang_ANT}.

Now let $s$, $(x_1,x_2)$ and $\ell$ be as in the claim above. Let
$\varphi\colon\Zbar\to\Fbar_p$ be a morphism of rings. Then there are
unique embeddings of $O_{1,s}$ and $O_{2,s}$ into $\Zbar$ that
composed with $\varphi$ give the actions on the tangent spaces at~$0$
of $E_1$ and~$E_2$. Let $m$ be a maximal ideal of index $\ell$ in
$O_{1,s}O_{2,s}\subset\Zbar$, and $m_1$ and $m_2$ the intersections of
$m$ with $O_{1,s}$ and~$O_{2,s}$. By the facts recalled at the end of
\S~\ref{sec:2}, there are canonical $\tilde{x_1}$ and $\tilde{x_2}$ in
$\Zbar$ lifting $E_1$ and $E_2$ to $\tilde{E}_1$ and $\tilde{E}_2$
with $\End(\tilde{E}_1)=\End(E_1)$ and
$\End(\tilde{E}_2)=\End(E_2)$. Let $\sigma$ be a Frobenius element in
$G_{K_1K_2}$ at~$m$. Then $\tilde{E}_1=[m_1]^{-1}[m_1]\tilde{E}_1$
shows that $\tilde{E}_1$ is $\ell$-isogenous to $[m_1]\tilde{E}_1$
which is the conjugate of $\tilde{E}_1$ by~$\sigma^{-1}$, and
similarly for~$\tilde{E}_2$. Then $([m_1]E_1,[m_2]E_2)$ is the
reduction of $\sigma^{-1}(\tilde{E}_1,\tilde{E}_2)$, hence in
$s(\Fbar_p)$. So $(x_1,x_2)$ is in $(T_\ell\times T_\ell)
([m_1]E_1,[m_2]E_2)$. So $(x_1,x_2)$ is both in $s(\Fbar_p)$ and in
$(T_\ell\times T_{\ell})(s(\Fbar_p))$. We conclude that $s(\Fbar_p)$ is
contained in $Z(\Fbar_p)\cap (T_\ell\times T_{\ell})Z(\Fbar_p)$. Now
the degrees of the projections from $(T_\ell\times T_\ell)Z$ to
$\A^1_{\F_p}$ are $(\ell+1)^2d_1$ and $(\ell+1)^2d_2$, so the
intersection number (in $(\P^1\times\P^1)_{\F_p}$) of $Z$ and
$(T_\ell\times T_\ell)Z$ is $2d_1d_2(\ell+1)^2$. But the intersection
contains $s(\Fbar_p)$, which has more points than this intersection
number, so the intersection is not of dimension~$0$.
\end{proof}

\section{Goursat's lemma and  Zarhin's theorem}
\label{sec:zarhin}
Here we deviate from the topological approach in \cite{Edixhoven1}
and~\cite{Edixhoven2}.
\begin{Theorem}\label{thm:gl+Zt}
  Let $C$ be an irreducible reduced closed curve in $\A^2_{\F_p}$, not
  a fibre of one of the $2$ projections, such that there are
  infinitely many prime numbers $\ell$ for which $(T_\ell\times
  T_\ell)(C)$ is reducible. Then there is an $n\in\Z_{>0}$ such that
  $C$ is the image of an irreducible component of $Y_0(n)_{\F_p}$
  in~$\A^2_{\F_p}$.
\end{Theorem}
\begin{proof}
Let $K$ denote the function field of~$C$, and let $E_1$ and $E_2$ be
elliptic curves over $K$ with $j$-invariants the projections $\pi_1$
and $\pi_2$, viewed as functions on~$C$; these $E_1$ and
$E_2$ are unique up to quadratic twist. We must prove that $E_1$ is
isogeneous to a twist of~$E_2$.

Let $K\to K^\sep$ be a separable closure and let $G:=\Gal(K^\sep/K)$.
For $\ell\neq p$ a prime number, let $V_{\ell,1}:=E_1(K^\sep)[\ell]$
and $V_{\ell,2}:=E_2(K^\sep)[\ell]$ and let $G_\ell$ be the image of
$G$ in $\GL(V_{\ell,1})\times\GL(V_{\ell,2})$, with projections
$G_{\ell,1}$ and~$G_{\ell,2}$. Because of the Weil pairing, $G$ acts
on $\det(V_{\ell,1})$ and $\det(V_{\ell,2})$ by the cyclotomic
character $\chi_\ell\colon G\to\F_l^\times=\Aut(\mu_\ell(K^\sep))$.
For all but finitely many~$\ell$, $G_{\ell,1}$ contains
$\SL(V_{\ell,1})$ and similarly for~$E_2$ (this follows, as
in~\cite{Co-Ha}, from the fact that for $n$ prime to $p$ the geometric
fibres of the modular curve over $\Z[\zeta_n,1/n]$ parametrising
elliptic curves with symplectic basis of the $n$-torsion are
irreducible, \cite[Thm.~3]{Igusa} and \cite[Cor.~10.9.2]{Katz-Mazur}).
Let $q$ be the number of elements of the algebraic closure of $\F_p$
in~$K$. Then, for all but finitely many $\ell$, $G_{\ell,1}$ is the
subgroup of elements in $\GL(V_{\ell,1})$ whose determinant is a power
of $q$, and similarly for~$G_{\ell,2}$. Let $L$ be the set of prime
numbers $\ell\neq2$ for which $G_{\ell,1}$ and $G_{\ell,2}$ are as in
the previous sentence, and such that $(T_\ell\times T_\ell)(C)$ is
reducible. Then $L$ is infinite.

Let $\ell$ be in~$L$. Let $N_{\ell,1}:=\ker(G_\ell\to G_{\ell,2})$ and
$N_{\ell,2}:=\ker(G_\ell\to G_{\ell,1})$. Then $N_{\ell,i}$ is a
normal subgroup of $G_{\ell,i}\cap\SL(V_{\ell,i})$, and $G_\ell$ is
the inverse image of the graph of an isomorphism
$G_{\ell,1}/N_{\ell,1}\to G_{\ell,2}/N_{\ell,2}$. The only normal
subgroups of $\SL_2(\F_\ell)$ are the trivial subgroups and the center
$\{\pm 1\}$, with different number of elements. As $\# G_{\ell,1}=\#
G_{\ell,2}$, we have $\# N_{\ell,1}=\# N_{\ell,2}$, and so there are 3
cases. 

If $N_{\ell,1}=\SL(V_{\ell,1})$, then $G_\ell$ contains
$\SL(V_{\ell,1})\times\SL(V_{\ell,2})$, contradicting the reducibility
of $(T_\ell\times T_\ell)(C)$. Hence $N_{\ell,1}$ is $\{\pm 1\}$
or~$\{1\}$, and $G_\ell$ gives us an isomorphism 
$\varphi_\ell\colon G_{\ell,1}/\{\pm 1\}\to G_{\ell,2}/\{\pm 1\}$.
As all automorphisms of $\SL_2(\F_\ell)/\{\pm 1\}$ are induced by
$\GL_2(\F_\ell)$ (\cite{Schreier-vdW}, or \cite[\S 3.3.4]{Wilson}),
there is an isomorphism $\gamma\colon V_{\ell,1}\to V_{\ell,2}$ of
$\F_\ell$-vector spaces (not necessarily $G$-equivariant) that
induces the restriction $\varphi_\ell$ from $\SL(V_{\ell,1})/\{\pm
1\}$ to $\SL(V_{\ell,2})/\{\pm 1\}$. Let $\alpha_\ell$ be the
automorphism of $G_{\ell,1}/\{\pm 1\}$ obtained as the composition of
first $\varphi_\ell$ and then $G_{\ell,2}/\{\pm 1\}\to G_{\ell,1}/\{\pm 1\}$,
$g\mapsto \gamma^{-1}g\gamma$. Consider the short exact
sequence 
\[
\{1\}\to \SL(V_{\ell,1})/\{\pm 1\} \to G_{\ell,1}/\{\pm 1\} \to
\langle q\rangle \to \{1\}\,.
\]
Then $\alpha_\ell$ induces the identity on $\SL(V_{\ell,1})/\{\pm 1\}$
and on~$\langle q\rangle$. Lemma~\ref{lem:autom-of-ext} gives us that
$\alpha_\ell$ is the identity. Hence $\varphi_\ell$ is the morphism
$G_{\ell,1}/\{\pm 1\}\to G_{\ell,2}/\{\pm 1\}$, $g\mapsto \gamma
g\gamma^{-1}$. If $N_{\ell,1}=\{1\}$ then $G_\ell$ is
$\Gamma_{\ell,\gamma}:=\{(g,\gamma g\gamma^{-1}) : g\in G_{\ell,1}\}$,
and if $N_{\ell,1}=\{\pm1\}$ then $G_\ell$ is
$\Gamma^\pm_{\ell,\gamma}:= \{(g,\pm\gamma g\gamma^{-1}) : g\in
G_{\ell,1}\}$. This means that $\gamma\colon V_{\ell,1}/\{\pm 1\}\to
V_{\ell,2}/\{\pm 1\}$ is $G$-equivariant. Even better, writing, for
$g$ in $G$, $\gamma(gv)=\eps_\ell(g)g(\gamma(v))$ with
$\eps_\ell(g)\in \{\pm 1\}$, this $\eps_\ell\colon G\to
\{\pm1\}\subset\F_l^\times$ is a character, and $\gamma$ is an
isomorphism from $V_{\ell,1}$ to the twist of $V_{\ell,2}$
by~$\eps_\ell$.

Let $U\subset C$ be the open subscheme where $C$ is regular and where
$E_1$ and $E_2$ have good reduction. Then for all $\ell$ in $L$, and
all closed $x$ in $U$, $\eps_\ell$ is unramified at~$x$. As $U$ is a
smooth curve over a finite field, there are only finitely many
characters $\eps\colon G\to\{\pm1\}$ unramified on~$U$, if $p\neq 2$
(this uses Kummer theory). For $p=2$, one has to be more careful; we
argue as follows. There are infinitely many characters $\eps\colon
G\to\{\pm1\}$ unramified on~$U$, but only finitely many with bounded
conductor on the projective smooth curve $\ol{C}$ with function
field~$K$. Let $K'\subset K^\sep$ be the extension cut out by
$V_{3,1}\times V_{3,2}$, and let $\ol{C}'\to \ol{C}$ be the
corresponding cover. Then both $E_1$ and $E_2$ have semistable
reduction over~$\ol{C}'$ by \cite[Cor.~5.18]{Desch}. The Galois
criterion for semi-stability in \cite[Exp.~IX, Prop.~3.5]{Groth} tells
us that all $\eps_\ell$ become unramified on~$\ol{C}'$. This shows
that also for $p=2$ there are only finitely many distinct~$\eps_\ell$.
The conclusion is that, for general~$p$, there are only finitely many
distinct $\eps_\ell$, and therefore we can assume (by shrinking $L$ to
an infinite subset) that they are all equal to some~$\eps$. Then we
replace $E_2$ by its twist by $\eps$, and then $\eps_\ell$ are
trivial.

Now Zarhin result \cite[Cor.~2.7]{Zarhin} tells us that there is a
non-zero morphism $\alpha\colon E_1\to E_2$. 
\end{proof}

\begin{Remark}
Up to sign, there is a unique isogeny $\alpha\colon E_1\to E_2$ of
minimal degree. Then $C$ is an irreducible component of the image of
$Y_0(n)_{\F_p}$. We write $n=p^km$ with $m$ prime to~$p$. Then $C$ is
the image of the image of $Y_0(m)_{\F_p}$ by the $p^k$-Frobenius map on
the first or on the second coordinate, and $C$ is also an irreducible
component of the images of all $Y_0(p^{2i}n)$ with $i\in\Z_{\geq0}$.
\end{Remark}

\begin{Lemma}\label{lem:autom-of-ext}
Let $G$ be a group, $N$ a normal subgroup of $G$ and $Q$ the
quotient. Let $\alpha$ be an automorphism of $G$ inducing the identity
on $N$ and on $Q$, and suppose that $G$ acts trivially by conjugation
on the center of $N$, and that there is no non-trivial morphism
from $Q$ to the center of~$N$. Then $\alpha$ is the identity on~$G$.
\end{Lemma}
\begin{proof}
We write, for all $g\in G$: 
\[
\alpha(g) = g\beta(g),\quad
\text{with $\beta$ a map (of sets!) from $G$ to itself.}
\]
As $\alpha$ induces the identity on $Q$, $\beta$ takes values
in~$N$. As $\alpha$ is the identity on $N$, we have $\beta(n)=1$ for
all $n\in N$. For all $g_1$ and $g_2$ in $G$ we have:
\[
g_1g_2\beta(g_1g_2)=\alpha(g_1g_2)=\alpha(g_1)\alpha(g_2)=
g_1\beta(g_1)g_2\beta(g_2)\,,
\]
and therefore
\[
\beta(g_1g_2) = g_2^{-1}\beta(g_1)g_2\beta(g_2)\,.
\]
For $g_1$ in $N$, this gives that for all $g_2$ in $G$,
$\beta(g_1g_2)=\beta(g_2)$. Hence $\beta$ factors through
$\ol{\beta}\colon Q\to N$: $\beta(g)=\ol{\beta}(\ol{g})$. Now, for
$g_1$ in $G$ and $g_2$ in $N$, we have
\[
\ol{\beta}(\ol{g_1}) = \ol{\beta}(\ol{g_1g_2}) = g_2^{-1}\ol{\beta}(\ol{g_1})g_2\,,
\]
hence $\ol{\beta}$ takes values in the center of~$N$. Now let $g_1$
and $g_2$ be in~$G$. As $g_2^{-1}\beta(g_1)g_2=\beta(g_1)$,
$\ol{\beta}$ is a morphism of groups from $Q$ to the center of $N$ and
therefore trivial.
\end{proof}
\section{Proof of the main theorem}
We are now ready to prove Theorem~\ref{Theorem_main}. 

If $Z=\A^2_{\F_p}$ or is finite, then $Z$ has no irreducible
components of dimension~$1$. Now assume that $Z$ has dimension~$1$. We
write $Z=V\cup H\cup F\cup Z'$ with $V$ the union in~$Z$ of fibers of
the 1st projection $\pr_1$, $H$ the union in~$Z$ of fibers of~$\pr_2$,
and $F$ the set of isolated points in~$Z$, and $Z'$ the union of the
remaining irreducible components of~$Z$. Let $B_1$ be the image of
$V\cup F$ under~$\pr_1$, and $B_2$ the image of $H\cup F$
under~$\pr_2$. 

Let $s$ be in $\Sigma$ such that $\pr_1(s)$ meets~$B_1$. Then either
$\pr_1(s)(\Fbar_p)$ consists of supersingular points, or it consists
of ordinary points with the same endomorphism ring as an ordinary point
in~$B_1$. Hence for such a $\pr_1(s)$ there are only finitely many
possibilities. Similarly for the~$\pr_2(s)$. It follows that the $s$
in $\Sigma$ with $\pr_1(s)$ disjoint from $B_1$ and $\pr_2(s)$
disjoint from $B_2$ are contained in~$Z'$. Let $\Sigma'$ be the set of
these~$s$. The $s$ in $\Sigma-\Sigma'$ lie on a finite union of fibres
of $\pr_1$ and~$\pr_2$, and the intersection of this union with $Z'$
is finite. Therefore the union of the $s$ in $\Sigma'$ is dense
in~$Z'$. We replace $Z$ by $Z'$, and $\Sigma$ by~$\Sigma'$. Then all
irreducible components of $Z$ are of dimension $1$ and are not a fibre
of $\pr_1$ or~$\pr_2$. Let $d_i$ ($i$ in $\{1,2\}$) be the degree of
$\pr_i$ restricted to~$Z$.

There are only finitely many points $(x_1,x_2)$ in $Z(\Fbar_p)$ such
that $x_1$ or $x_2$ is not ordinary. Therefore we can replace $\Sigma$
by its subset of $s$'s whose image under both projections is ordinary.

Theorem~\ref{thm:primes} gives us an infinite set $L$ of primes $\ell$
such that $Z\cap(T_\ell\times T_\ell)Z$ is of dimension~$1$. Let
$(Z_i)_{i\in I}$ be the set of irreducible components of~$Z$. Then for
each $\ell$ in $L$ there are $i$ and $j$ in $I$ such that $Z_i$ is in
$(T_\ell\times T_\ell)Z_j$. If moreover $\ell>12d_1$ then $(T_\ell\times
T_\ell)Z_j$ is reducible, because if not, then $(T_\ell\times
T_\ell)Z_j$ equals $Z_i$ (as closed subsets of~$\A^2_{\F_p}$), but for
any ordinary $(x,y)$ in $Z_j(\Fbar_p)$, $T_\ell(y)$ consists of at
least $(\ell+1)/12>d_1$ distinct points.

There is a $j_0\in I$ such that for infinitely many $\ell\in L$,
$(T_\ell\times T_\ell)Z_{j_0}$ is reducible. Theorem~\ref{thm:gl+Zt}
then tells us that there is an $n\geq1$ such that $Z_{j_0}$ is the
image in $\A^2_{\F_p}$ of an irreducible component of~$Y_0(n)_{\F_p}$.
We let $T(n)$ be the reduced closed subscheme of $\A^2_\Z$ whose
geometric points correspond to pairs $(E_1,E_2)$ of elliptic curves
that admit a morphism $\varphi\colon E_1\to E_2$ of degree~$n$.  Let
$J$ be the set of $j\in I$ such that $Z_j$ is an irreducible component
of~$T(n)_{\F_p}$, let $Z(n)$ be their union, and and let $Z'$ be the
union of the $Z_i$ with $i\not\in J$.

We claim that any $s$ in $\Sigma$ that meets $T(n)_{\F_p}$ is
contained in~$T(n) _{\F_p}$. So let $(j(E_1),j(E_2))$ be in
$s(\Fbar_p)$, and $\varphi\colon E_1\to E_2$ of degree~$n$. Let
$\Zbar\to\Fbar_p$ be a morphism of rings, and
$\tilde{\varphi}\colon\tilde{E}_1\to\tilde{E}_2$ be the canonical lift
over~$\Zbar$. Then $\tilde{\varphi}$ is of degree~$n$, and so are all
its conjugates by~$G_\Q$, and so $s(\Fbar_p)$, consisting of all
reductions of these conjugates, lies in $T(n)(\Fbar_p)$. 

As $T(n)_{\F_p}\cap Z'$ is finite, the set $\Sigma'$ of $s$ in
$\Sigma$ that do not meet $T(n)_{\F_p}$ is dense in $Z'$ and our proof
is finished by induction on the number of irreducible components of~$Z$.

\begin{Remark}
  We think that Theorem~\ref{Theorem_main} remains true if $E\subset
  \Qbar$ is a finite extension of~$\Q$ and we work with
  $\A^2_{\Fbar_p}$ and consider reductions of $G_E$-orbits of
  CM-points in $\A^2(\Zbar)$. However, the case $E=\Q$ has a special
  feature: up to fibres of the projections, the $Z$ are invariant
  under switching the coordinates. This comes from the dihedral nature
  of the Galois action. As soon as $E$ contains an imaginary quadratic
  field, there are $\Sigma$ such that $Z$ consists of one irreducible
  component of~$Y_0(p)_{\F_p}$.
\end{Remark}

\end{document}